\documentclass[11pt]{amsart}

\usepackage{amssymb,amsfonts}
\usepackage{amsaddr}
\usepackage{amsmath}
\usepackage{amsthm}
\usepackage{enumerate}
\usepackage{thmtools}
\usepackage{datetime}
\usepackage{graphicx}
\usepackage{pdfpages}
\usepackage{bbm}
\definecolor{darkspringgreen}{rgb}{0.14, 0.7, 0.3}
\definecolor{melon}{rgb}{0.4, 0.2, 1}
\usepackage[colorlinks, citecolor = darkspringgreen, linkcolor = melon]{hyperref}

\calclayout

\setboolean{@twoside}{false}

\newtheorem{thm}{Theorem}[section]
\newtheorem{cor}[thm]{Corollary}

\newtheorem{lem}[thm]{Lemma}

\theoremstyle{definition}
\newtheorem{defn}[thm]{Definition}

\newtheorem{exmp}[thm]{Example}

\theoremstyle{remark}
\newtheorem{rem}[thm]{Remark}

\newtheorem{obs}[thm]{Observation}

\makeatletter
\let\c@equation\c@thm
\makeatother
\numberwithin{equation}{section}

\makeatletter
\newcommand*\bigcdot{\mathpalette\bigcdot@{.5}}
\newcommand*\bigcdot@[2]{\mathbin{\vcenter{\hbox{\scalebox{#2}{$\m@th#1\bullet$}}}}}
\makeatother

\makeatletter
\def\subsection{\@startsection{subsection}{3}%
  \z@{.5\linespacing\@plus.7\linespacing}{.1\linespacing}%
  {\bfseries}}
\makeatother

\makeatletter
\def\@tocline#1#2#3#4#5#6#7{\relax
  \ifnum #1>\c@tocdepth 
  \else
    \par \addpenalty\@secpenalty\addvspace{#2}%
    \begingroup \hyphenpenalty\@M
    \@ifempty{#4}{%
      \@tempdima\csname r@tocindent\number#1\endcsname\relax
    }{%
      \@tempdima#4\relax
    }%
    \parindent\z@ \leftskip#3\relax \advance\leftskip\@tempdima\relax
    \rightskip\@pnumwidth plus4em \parfillskip-\@pnumwidth
    #5\leavevmode\hskip-\@tempdima
      \ifcase #1
       \or\or \hskip 1em \or \hskip 2em \else \hskip 3em \fi%
      #6\nobreak\relax
    \hfill\hbox to\@pnumwidth{\@tocpagenum{#7}}\par
    \nobreak
    \endgroup
  \fi}
\makeatother

\newcommand{\N}{\mathbb{N}}

\newcommand{\Z}{\mathbb{Z}}
\newcommand{\calA}{\mathcal{A}}
\newcommand{\calC}{\mathcal{C}}
\newcommand{\calD}{\mathcal{D}}
\newcommand{\calF}{\mathcal{F}}
\newcommand{\calP}{\mathcal{P}}
\newcommand{\Th}{\text{th}}
\DeclareMathOperator{\wt}{wt}

\DeclareMathOperator{\sh}{sh}
\DeclareMathOperator{\Trop}{Trop}

\title{On Dyck Path Expansion Formulas for Rank 2 Cluster Variables}
\author{Amanda Burcroff}

\begin{document}

\begin{abstract}
    In this paper, we simplify and generalize formulas for the expansion of rank $2$ cluster variables.  In particular, we prove an equivalent, but simpler, description of the colored Dyck subpaths framework introduced by Lee and Schiffler.  We then prove the conjectured bijectivity of a map constructed by Feiyang Lin between collections of colored Dyck subpaths and compatible pairs, objects introduced by Lee, Li, and Zelevinsky to study the greedy basis.  We use this bijection along with Rupel's expansion formula for quantum greedy basis elements, which sums over compatible pairs, to provide a quantum generalization of Lee and Schiffler's colored Dyck subpaths formula.
\end{abstract}

\maketitle

\tableofcontents

\section{Introduction}

The theory of \emph{cluster algebras}, introduced twenty years ago by Fomin and Zelevinsky \cite{FZ},  gives us a combinatorial framework for understanding the previously opaque nature of certain algebras.  Each cluster algebra is generated by its \emph{cluster variables}, which can be obtained via the recursive process of \emph{mutation}.  The \emph{Laurent phenomenon} says that each cluster variable in a rank-$n$ cluster algebra can be expressed as a Laurent polynomial in the $n$ \emph{initial} cluster variables.  While in general finding explicit formulas for the Laurent expansions of arbitrary cluster variables is difficult, there has been significant progress in understanding the expansions in low-rank cluster algebras.  In this work, we attempt to unify and simplify some of the existing expansion formulas for rank-$2$ cluster variables and their quantum generalizations.  


In 2011, Lee and Schiffler provided the first combinatorial formula for the Laurent expansion of arbitrary skew-symmetric rank-$2$ cluster variables \cite[Theorem 9]{LS}.  They expressed the coefficients as sums over certain collections of non-overlapping \emph{colored subpaths} of a \emph{maximal Dyck path}.  This established the positivity of the Laurent expansion in skew-symmetric rank $2$ cluster algebras.  Lee and Schiffler \cite[Theorem 11]{LS2} and Rupel \cite[Theorem 6]{Rup} then generalized this formula (in the skew-symmetric and skew-symmetrizable cases, respectively) to the non-commutative rank-2 setting, giving each collection a weight expressed as an ordered product of two non-commuting initial cluster variables.  In 2012, Lee, Li, and Zelevinsky \cite{LLZ} defined the \emph{greedy basis} for rank-$2$ cluster algebras, which includes the cluster variables.  They provided a combinatorial formula for the Laurent expansion of each greedy basis elements as a sum over \emph{compatible pairs}, certain collections of edges of a maximal Dyck path \cite[Theorem 11]{LLZ}. Rupel later gave a non-commutative analogue of this formula, which specializes to a formula for the coefficients in the \emph{quantum rank-$2$ cluster algebra} setting \cite[Corollary 5.4]{Rup2}.  In particular, each compatible pair is weighted by a corresponding power of a quantum parameter $q$, where the exponent is computed as a sum over all pairs of edges in the maximal Dyck path.  Rupel \cite[Theorem 1.2]{Rup3} has also provided a quantum analogue to the Caldero-Chapoton expansion formula \cite{CC} for rank-$2$ cluster variables expressed as a sum over indecomposable valued-quiver representations.

To summarize, there are two different combinatorial formulas for the Laurent expansion of skew-symmetric rank-$2$ cluster variables: one in terms of collections of colored Dyck subpaths \cite{LS} and the other in terms of compatible pairs \cite{LLZ}.  The combinatorics of collections of colored Dyck subpaths and compatible pairs are somewhat similar, suggesting that there might be a nice correspondence between them.  A correspondence was known to Lee, Li, and Zelevinsky \cite{LLZ} and they suggested in their 2012 paper that they planned to provide details in the future, but this has not yet appeared in the literature.  In 2021, Feiyang Lin constructed a map between a superset of the collections of colored Dyck subpaths and compatible pairs and conjectured that the map restricts to a bijection in \cite[Conjecture 3]{Lin}.  Lin made partial progress toward proving this, reducing the conjecture to a technical statement \cite[Conjecture 4]{Lin}.

In this work, we start by providing a simplification of Lee-Schiffler's formula for rank-$2$ cluster variables in terms of colored Dyck subpath conditions.  We then use our simpler formula to prove Lin's conjectures \cite[Conjectures 3 \& 4]{Lin} that the map constructed between collections of colored Dyck subpaths and compatible pairs is indeed a bijection.  (Our methods do not rely on the technical reformulation presented by Lin.)  This bijection gives an efficient method for generating all compatible pairs in the cluster variable case.  We then use the bijection along with Rupel's quantum weighting of compatible pairs \cite{Rup2} to provide a quantum version of Lee and Schiffer's rank-$2$ expansion formula for cluster variables.  This new formula has the advantage of requiring less computation than that in \cite[Corollary 5.4]{Rup2} and explicitly calculating the coefficients in the quantum case, rather than expressing each term as an ordered product as in \cite{Rup}.  It is also more elementary than the expansion formula in \cite{Rup3}, which is based on the theory of valued quiver representations.  

The paper is organized as follows.  In \autoref{sec: results}, we give an overview of the results.  \autoref{sec: prelims} contains some preliminaries concerning maximal Dyck paths.   The proof of the simplification of Lee and Schiffler's \cite{LS} colored Dyck subpath conditions is the focus of \autoref{sec: simplification}.  \autoref{sec: bijection} contains the proof of Lin's conjectures \cite[Conjectures 3 \& 4]{Lin}, establishing a bijection between collections of colored Dyck subpaths from the Lee-Schiffler \cite{LS} setting and compatible pairs from the Lee-Li-Zelevinsky \cite{LLZ} setting.  This bijection is applied to Rupel's \cite{Rup2} quantum weighting on compatible pairs to yield a quantum analogue of Lee-Schiffler's \cite{LS} expansion formula in \autoref{sec: quantum}.  We conclude with a discussion of further directions in \autoref{sec: further directions}.

\section{Statement of Results}\label{sec: results}
For a positive integer $r$ and variables $X_1,X_2$, we consider the sequence $\{X_n\}_{n \in \Z}$ of expressions recursively defined by 
\begin{equation}\label{eqn: cluster recurrence}
    X_{n+1} = \frac{X_n^r + 1}{X_{n-1}}\,.
\end{equation}
This sequence is precisely the set of variables of the rank-$2$ cluster algebra $\calA(r,r)$ associated to the $r$-Kronecker quiver, which consists of two vertices with $r$ arrows between them.  The sequence is periodic when $r = 1$, and otherwise all $X_n$ are distinct.  For background on cluster algebras, see \cite{FWZ}.

\begin{defn}\label{defn: maximal Dyck path}
The \emph{maximal Dyck path} $\calP(a,b)$ is the path proceeding by unit north and east steps from $(0,0)$ to $(a,b)$ that is closest to the line segment between $(0,0)$ and $(a,b)$ without crossing strictly above it.  For two vertices $u,w$ along such a path, let $s(u,w)$ denote the slope of the line segment between them. 
\end{defn}

 Let $\{c_n\}_{n=1}^\infty$ be the sequence of non-negative integers defined recursively by:
\begin{equation}\label{eqn: cn recurrence}
    c_{1} = 0, c_{2} = 1, \text{ and } c_{n} = rc_{n-1} - c_{n-2} \text{ for } r \geq 2\,.
\end{equation}
Let $\calC_n = \calP(c_{n-1},c_{n-2})$ and $\calD_n = \calP(c_{n-1}-c_{n-2},c_{n-2})$.  We label the leftmost vertex at each height of $\calD_n$ by $v_i$ for $i = 0,\dots,c_{n-2}$, with the subindex increasing from south to north; such vertices are called \emph{northwest corners}.  Let $\gamma(i,k)$ be the subpath spanning from $v_i$ to $v_k$ for any $0 \leq i < k \leq b$.  An example of the maximal Dyck path $\calD_5 = \calP(5,3)$ is shown in \autoref{fig: intro exmp}.

\begin{thm}\label{thm: path slope result}
    Let $t(i)$ be the minimum integer greater than $i$ such that $s(v_i,v_{t(i)}) > s$.  Then we have $t(i) - i = c_m - wc_{m-1}$ for a unique choice of $2 \leq w \leq r-1$ and $3 \leq m \leq n-2$.  
\end{thm}

This result allows us to simplify the expansion formula of Lee and Schiffler, which is briefly described below. The next few definitions and \autoref{cor: LS modified expansion} emulate the results of Lee and Schiffler, except for slight modifications due to the simplified coloring conditions above. 

\begin{defn}[cf. \autoref{def: LS colors}]\label{def: brown path}
For any $0 \leq i < k \leq c_{n-2}$, let $\gamma(i,k)$ be the subpath of $\calD_n$ from $v_i$ to $v_k$, which is assigned a color as follows: 
\begin{enumerate}
    \item[(1)] If $s(v_i,v_t) \leq s$ for all $t$ such that $i < t \leq k$, then $\gamma(i,k)$ is {\color{blue} blue}.
    \item[($2^*$)] Otherwise, let $m,w$ be chosen with respect to $i$ as in \autoref{thm: path slope result}.  Then we say $\gamma(i,k)$ is {\color{brown} $(m,w)$-brown}.\footnote{The color brown was chosen because it is the combination of Lee and Schiffler's red and green cases.}
\end{enumerate}
\end{defn}

A \emph{subpath} of $\calD_n$ is a path of the form $\gamma(i,k)$ or a single edge $\alpha_i$.  We denote the set of such subpaths by $\calP'(\calD_n)$. We define the set $\calF'(\calD_n)$ to contain any collection of subpaths in $\calP'(\calD_n)$ satisfying that no two subpaths share an edge, two subpaths share a vertex only if at least one of them is a single edge, and at least one of the $c_{m-1}-2c_{m-2}$ edges preceding each $(m,w)$-brown subpath is contained in another subpath.   Given $\beta \in \calF'(\calD_n)$, the quantity $|\beta|_1$ is defined additively over the subpaths, taking value $k-i$ on $\gamma(i,k)$ and value $0$ on single edges.  The quantity $|\beta|_2$ is the total number of edges in $\beta$. This yields the following expansion formula for the cluster variables.

\begin{cor}[{\text{analogue of }\cite[Theorem 9]{LS}}]\label{cor: LS modified expansion}
 Consider the cluster algebra $\calA(r,r)$ with cluster variables $X_i$ for $i \in \Z$.  For $n \geq 4$, we have
 $$X_n = X_1^{-c_{n-1}}X_2^{-c_{n-2}} \sum_{\beta \in \calF'(\calD_n)} X_1^{r|\beta|_1}X_2^{r\left(c_{n-1} - |\beta|_2\right)}$$
 and 
 $$X_{3-n} = X_2^{-c_{n-1}}X_1^{-c_{n-2}} \sum_{\beta \in \calF'(\calD_n)} X_2^{r|\beta|_1}X_1^{r\left(c_{n-1} - |\beta|_2\right)}\,.$$
\end{cor}

A generalization of the above expansion result to the case of skew-symmetric rank-$2$ cluster algebras with coefficients is presented at the end of \autoref{subsec: simplification proof} (see \autoref{cor: LS modified expansion coeffs}).

\begin{exmp}\label{exmp: beta intro}
    The collection of colored subpaths $\beta = \{\gamma(0,1),\alpha_6,\gamma(2,3)\}$ is in $\calF'(\calD_5)$.  The subpath $\gamma(0,1)$ is blue, while the subpath $\gamma(2,3)$ is $(3,2)$-brown.  This collection is depicted in \autoref{fig: intro exmp} for $r = 3$, where the single edge $\alpha_6$ is represented by an orange edge. Note that in this case we have $|\beta|_1 = 2$ and $|\beta|_2 = 6$.
\end{exmp}

We now discuss the bijection between colored Dyck subpaths of $\calD_n$ and compatible pairs in $\calC_n$.  Given two vertices $u,w$ in a maximal Dyck path $\calP(a,b)$, let $\overrightarrow{uw}$ denote the subpath proceeding east from $u$ to $w$, continuing cyclically around $\calP(a,b)$ if $u$ is to the east of $w$.  Let $|uw|_1$ (resp., $|uw|_2$) denote the number of horizontal (resp., vertical) edges of $\overrightarrow{uw}$.  Given a set of horizontal edges $S_1$ and vertical edges $S_2$ in $\calP(a,b)$, the pair $(S_1,S_2)$ is \emph{compatible} if, for every edge in $S_1$ with left vertex $u$ and every edge $S_2$ with top vertex $w$, there exists a lattice point $t \neq u,w$ in the subpath $\overrightarrow{uw}$ such that 
$$|tw|_1 = r|\overrightarrow{tw} \cap S_2|_2 \text{ or } |ut|_2 = r|\overrightarrow{ut} \cap S_1|_1\,.$$

Let the horizontal (resp., vertical) edges of $\calC_n$ be labeled by $\eta_i$ (resp., $\nu_i$), increasing to the east (resp., north).  Lin defined the following map $\Phi$ from $\calF'(\calD_n)$ to pairs $(S_1,S_2)$ in $\calC_n$.  Note that while we define $\Phi$ as a map on blue/brown colored subpaths, it was originally defined for Lee-Schiffler's blue/green/red colored subpaths, and the two definitions are essentially identical.

\begin{figure}
    \centering
    \includegraphics[width = 6in]{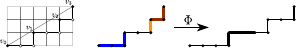}
    \caption{The leftmost image shows the maximal Dyck path $\calD_5$ for $r = 3$, along with the corresponding $5\times 3$ grid and main diagonal.  The northwest corners are labeled and depicted as filled vertices.  The center image is a collection $\beta$ of colored Dyck subpaths in $\calF'(\calD_5)$, and the rightmost image is the compatible pair on $\calC_5$ that $\beta$ maps to under $\Phi$, where an edge is thickened whenever it is included in the compatible pair.}
    \label{fig: intro exmp}
\end{figure}

\begin{defn}[\cite{Lin}]\label{defn: Phi}
Given $\beta \in \calF'(\calD_n)$, let $\Phi(\beta) = (\Phi_1(\beta),\Phi_2(\beta))$, where
\begin{align*}
\Phi_1(\beta) &= \{\eta_s : \alpha_s \text{ is not a part of any subpath of }\beta\}\,,\\
\Phi_2(\beta) &= \{\nu_s : \gamma(i,k) \in \beta \text{ for some } i < s \leq k\}\,.
\end{align*}
\end{defn}

\begin{exmp}
The compatible pair $\left(\{\eta_4,\eta_5\}, \{\nu_1,\nu_3\}\right)$ obtained by applying $\Phi$ to the collection of subpaths $\beta \in \calF'(\calD_5)$ from \autoref{exmp: beta intro} is shown in \autoref{fig: intro exmp}.
\end{exmp}

\begin{thm}\label{thm: Phi bijectivity}
The map $\Phi$ is a bijection between collections of colored subpaths in $\calF'(\calD_n)$ and compatible pairs in $\calC_n$.
\end{thm}

Switching to the quantum setting, we will now work inside the quantum torus\\ $\mathcal{T} := \Z[q^{\pm 1}]\langle Z_1^{\pm 1}, Z_2^{\pm 1} : Z_1Z_2 = q^2Z_2Z_1\rangle$. The quantum rank-$2$ $r$-Kronecker cluster algebra $\calA_q(r,r)$ is the $\Z[q^{\pm 1}]$ subalgebra of the skew field of fractions of $\mathcal{T}$ generated by the quantum cluster variables $\{Z_n\}_{n \in \Z}$, which follow the recursion $Z_{n+1}Z_{n-1} = q^{-r}Z_n^r + 1$ (cf. \autoref{eqn: cluster recurrence}).  We use the bijectivity of $\Phi$ along with Rupel's quantum weighting on compatible pairs \cite{Rup2} to construct a quantum weighting of collections of colored subpaths.  

For $\beta = \{\beta_1,\dots,\beta_t\} \in \calF'(\calD_n)$, where $\beta_i$ appears to the left of $\beta_{i+1}$, we define the set of \emph{complimentary subpaths} $\overline{\beta_0},\overline{\beta_1},\dots,\overline{\beta_t}$ such that $\overline{\beta_i}$.  For $1 \leq i \leq t-1$, let $\overline{\beta_i}$ contain all edges and northwest corners between the end of $\beta_i$ and the start of $\beta_{i+1}$, where a northwest corner on the boundary of a path is included in $\overline{\beta_i}$ unless it is the right endpoint of a brown or blue subpath.   We set $\overline{\beta_0}$ to be the portion of the path before $\beta_1$ excluding $v_0$, and we set $\overline{\beta_t}$ to be the portion of the path after $\beta_t$ including $v_{c_{n-2}}$.  Note that it is possible for some $\overline{\beta_i}$ to be empty.  Let $|\overline{\beta_i}|_1$ (resp., $|\overline{\beta_i}|_2$) denote the number of northwest corners (resp., edges) in $\overline{\beta_i}$.

\begin{defn}\label{def: quantum weight}
For $\beta = \{\beta_1,\dots,\beta_t\} \in \calF'(\calD_n)$, we let
$$w_q(\beta) = (c_{n-1}+c_{n-2} - 1) + \sum_{j=0}^{t} r|\overline{\beta_j}|_2\left(\sum_{i=1}^t (-1)^{\mathbbm{1}_{i < j}}|\beta_i|_2 \right) + \left(r|\overline{\beta_j}|_1 - r^2|\overline{\beta_j}|_2\right)\left(\sum_{i=1}^t (-1)^{\mathbbm{1}_{i < j}}|\beta_i|_1 \right),$$
where $\mathbbm{1}_{i < j}$ takes value $1$ when $i < j$ and $0$ otherwise.  We then set
$$u_q(\beta) = w_q(\beta) - (c_{n-1}+c_{n-2} - 1) + (c_{n-1} - r|\beta|_1)(c_{n}-r|\beta|_2)\,.$$
\end{defn}

\begin{exmp}
For the collection $\beta \in \calF'(\calD_5)$ from \autoref{exmp: beta intro}, we have $\overline{\beta_0} = \overline{\beta_3} = \emptyset$, $\overline{\beta_1} = \{\alpha_4, \alpha_5\}$, and
$\overline{\beta_2} = \{v_2\}$.  We thus have $w_q(\beta) = 10$.
\end{exmp}

We prove that the quantum cluster variable Laurent coefficients can be expressed as a sum over weighted collections of subpaths in $\calF'(\calD_n)$.

\begin{thm}\label{thm: quantum expansion}
 Consider the quantum cluster algebra $\calA_q(r,r)$ with quantum cluster variables $Z_i$ for $i \in \Z$.  For $n \geq 4$, we have
 $$Z_n = Z_1^{-c_{n-1}}Z_2^{-c_{n-2}} \sum_{\beta \in \calF'(\calD_n)} q^{u_q(\beta)}Z_1^{r|\beta|_1}Z_2^{r\left(c_{n-1} - |\beta|_2\right)}$$
 and 
 $$Z_{3-n} = Z_2^{-c_{n-1}}Z_1^{-c_{n-2}} \sum_{\beta \in \calF'(\calD_n)} q^{u_q(\beta)}Z_2^{r|\beta|_1}Z_1^{r\left(c_{n-1} - |\beta|_2\right)}\,.$$
\end{thm}

\section{Preliminaries}\label{sec: prelims}

Let  $r \geq 2$ be fixed throughout this paper.  Both the Lee-Schiffler and Lee-Li-Zelevinsky expansion formulas involve sums over certain collections of edges in a maximal Dyck path.  Moreover, the width and height of these Dyck paths have certain recursive properties that are used in the proofs of both formulas.  We begin by setting up a framework for studying these paths and describing their recursive behavior.

Recall the sequence $c_n$ defined recursively by \autoref{eqn: cn recurrence}.  While the indexing of the sequence $\{c_{n}\}_{n\geq1}$ is identical to the indexing in the work of Lee and Schiffler \cite{LS}, the indexing is shifted by one from that defined by Lin \cite{Lin}, i.e., it is equivalent to $\{c_{n-1}\}_{n\geq 1}$ in Lin's work. It is straightforward to check that for $n > 1$, the quantities $c_{n}$ and $c_{n+1}$ are relatively prime, hence so are $c_{n}$ and $c_{n+1} - c_{n}$.  Thus the only vertices of $\calD_n$ and $\calC_n$ that lie on the main diagonal are the first and last.

Fix $a,b \in \N$.  Consider a rectangle with vertices $(0,0)$, $(0,b)$, $(a,0)$, and $(a,b)$ having a designated diagonal from $(0,0)$ to $(a,b)$.  

\begin{defn}
A \emph{Dyck path} is a lattice path in $\Z^2$ starting at $(0,0)$ and ending at a lattice point $(a,b)$ where $a,b \geq 0$, proceeding by only unit north and east steps and never passing strictly above the diagonal. Given a Dyck path $P$, we denote the number of east steps by $|P|_1$ and the number of north steps by $|P|_2$.  The \emph{length} of the Dyck path $P$ is the quantity $|P|_1 + |P|_2$.  We denote the set of lattice points contained in the Dyck path $P$, including the left and right endpoints, by $V(P)$.
\end{defn} 

The Dyck paths from $(0,0)$ to $(a,b)$ form a partially ordered set by comparing the heights at all vertices.  The maximal Dyck path $\calP(a,b)$, as defined in \autoref{defn: maximal Dyck path}, is the maximal element under this partial order.  We focus on the following two classes of maximal Dyck paths, defined for $n \geq 3$, $\calC_n = \calP(c_{n-1}, c_{n-2})$ and $\calD_n = \calP(c_{n-1} - c_{n-2}, c_{n-2})$.

Recall that a vertex of a maximal Dyck path $P$ is a \emph{northwest corner} if there are no vertices directly north of (equivalently, to the east of) it.  In $\calD_n$, these are precisely the vertices labeled by $v_i$ for some $0 \leq i \leq c_{n-2}$.

\begin{figure}
    \centering
    \includegraphics[width = 4in]{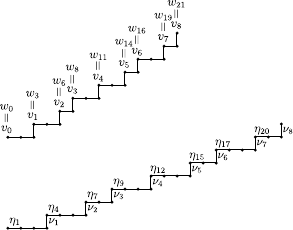}
    \caption{The maximal Dyck paths $\calD_6$ (above) and $\calC_6$ (below) are shown with some of their vertex and edge labels for $r = 3$.  Each northwest corner of $\calD_6$ is labeled with both its corresponding $w_i$ and $v_j$ label.  Some edges of $\calC_6$ are labeled; the $\nu_i$'s refer to the vertical edge left of the label, and the $\eta_j$'s refer to the horizontal edge below the label.}
    \label{fig: path label defns}
\end{figure}

When $a$ and $b$ are relatively prime, as is the case for $\calC_n$ and $\calD_n$, we can associate to this Dyck path the \emph{(lower) Christoffel word} of slope $\frac{b}{a}$ on the alphabet $\{E,N\}$.  This word can be constructed by reading the edges of the maximal Dyck path from $(0,0)$ to $(a,b)$, recording an $E$ for each east step and an $N$ for each north step.  For further background on Christoffel words, see \cite{BLRS}.  

\begin{exmp}
Let $r = 3$.  The Christoffel words corresponding to the maximal Dyck paths $\calD_6$ and $\calC_6$ depicted in \autoref{fig: path label defns} are $$E^2NE^2NENE^2NE^2NENE^2NEN \text{ and } E^3NE^3NE^2NE^3NE^3NE^2NE^3NE^2N\,,$$ respectively.
\end{exmp}

\begin{rem}\label{rem: C to D} The Christoffel word corresponding to $\calC_{n}$ is obtained by applying the morphism $\theta = \{N \mapsto EN\}$, i.e., the map that replaces each instance of the letter $N$ with the string $EN$, to the Christoffel word corresponding to $\calD_{n}$.  This follows directly from, for example, \cite[Lemma 2.2]{BLRS}.  
\end{rem}

\begin{obs}\label{obs: small paths}
It is straightforward to calculate that
$$V(\calC_3) = \{(0,0),(1,0)\} \text{ and } V(\calC_4) = \{(0,0),(1,0),\dots, (r,0), (r,1)\}$$
\end{obs}

As we shall see, both of these families of Dyck paths have a recursive structure.  The following lemma is a special case of a result of Rupel.

\begin{lem}[{\cite[Lemma 3]{Rup}}]\label{lem: vertex Dn recursion}
For all $n \geq 4$, the maximal Dyck path $\calD_n$ consists of $r-1$ copies of $D_{n-1}$ followed by a copy of $D_{n-1}$ with a prefix $D_{n-2}$ removed.  In particular, $\calD_{n-1}$ (resp., $\calC_{n-1}$) is a subpath of $\calD_{n}$ (resp., $\calC_{n}$).
\end{lem}

This allows us to define the limit of these paths, which can be realized by taking a union of finite subpaths.

\begin{defn}
Let $\calC$ (resp., $\calD$) be the infinite path on $\Z^2$ formed by the union $\bigcup_{n \geq 3} \calC_n$ (resp., $\bigcup_{n \geq 3} \calD_n$).  We identify the paths $\calC_n$ and $\calD_n$ with the prefix of the same length of $\calC$ and $\calD$, respectively.  Thus, the vertices of $\calD$ are labeled by $w_i$ and the northwest corners by $v_j$, as described after \autoref{defn: maximal Dyck path} in \autoref{sec: results}.  Similarly, horizontal edges of $\calC$ are labeled by $\eta_i$ and the vertical edges are labeled by $\nu_j$.
\end{defn}

\section{Simplification of the Colored Dyck Subpaths Conditions}\label{sec: simplification}
\subsection{Lee-Schiffler Expansion Formula}\label{subsec: ls expansion}
We first recall the original expansion formula given by Lee and Schiffler \cite{LS} for rank-two skew-symmetric cluster variables.  This requires us to set up the language of colored subpaths in a Dyck path via Lee and Schiffler's conventions, which differs from that in \autoref{sec: results}.  We then describe the map between certain non-overlapping collections of colored subpaths, namely from $\calF(\calD_n)$ as defined by Lee-Schiffler to the set $\calF'(\calD_n)$ which we defined after \autoref{def: brown path} in \autoref{sec: results}.

Let $s$ denote the slope of the main diagonal of $\calD_n$, so $s = \frac{c_{n-2,r}}{c_{n-1,r} - c_{n-2,r}}$.
\begin{defn}[\cite{LS}, cf. \autoref{def: brown path}]\label{def: LS colors}
 For any $0 \leq i < k \leq c_{n-2}$, let $\alpha(i,k)$ be the subpath of $\calD_n$ defined as follows:
 \begin{enumerate}[(1)]
     \item If $s(v_i,v_t) \leq s$ for all $t$ such that $i < t \leq k$, then $\alpha(i,k)$ is defined to be the subpath from $v_i$ to $v_k$; each such subpath is called {\color{blue} blue}.
     \item If $s(v_i,v_t) > s$ for some $i < t \leq k$, then
     \begin{enumerate}
         \item[(2-a)] if the smallest such $t$ is of the form $i + c_m - wc_{m-1}$ for some $3 \leq m \leq n-2$ and $1 \leq w \leq r-2$, then $\alpha(i,k)$ is defined to be the subpath from $v_i$ to $v_k$; each such subpath is called {\color{darkspringgreen} $(m,w)$-green}.
         \item[(2-b)] otherwise, $\alpha(i,k)$ is set to be the subpath from the vertex immediately below $v_i$ to $v_k$; each such subpath is called {\color{red} red}.
     \end{enumerate}
    \end{enumerate}
    \end{defn}
    Each such pair $i,k$ corresponds to precisely one subpath of $\calD_n$.  Denote the single edges of $\calD_n$ be  by $\alpha_1,\dots,\alpha_{c_{n-1}}$ proceeding from southwest to northeast, and let
    $$\calP(\calD_n) = \{\alpha(i,k) : 0 \leq i < k \leq c_{n-2}\} \cup \{\alpha_1,\dots,\alpha_{c_{n-1}}\}\,.$$
    The formula involves sums over collections of subsets of $\calP(\calD_n)$ satisfying certain non-overlapping requirements.  In particular, Lee and Schiffler set
    \begin{align*}
    \calF(\calD_n) = \{\{\beta_1,\dots,\beta_t\} :\;& t \geq 0, \beta_j \in \calP(\calD_n) \text{ for all } 1 \leq j \leq t,\\
    &\text{if $j \neq j'$ then $\beta_j$ and $\beta_{j'}$ have no common edge,}\\
    &\text{if $\beta_j = \alpha(i,k)$ and $\beta_{j'} = \alpha(i',k')$ then $i\neq k'$ and $i' \neq k$,}\\
    &\text{and if $\beta_j$ is $(m,w)$-green then at least one of the $(c_{m-1}-wc_{m-2})$}\\
    &\;\;\;\;\;\;\;\;\text{preceding edges of $v_i$ is contained in some $\beta_{j'}$}\}\,.
    \end{align*}
  
    For any collection of subpaths $\beta$, we associate two non-negative integers $|\beta|_1$ and $|\beta|_2$.  The first quantity $|\beta|_1$ is defined to be $0$ on single edges and $k-i$ on $\alpha(i,k)$, then extended additively on unions of these subpaths.  The second quantity $|\beta|_2$ is the total number of edges $\alpha_i$ covered by the subpaths in $\beta$.  We can now state the original formulation of Lee and Schiffler's expansion result.
    \begin{thm}[{\cite[Theorem 9]{LS}}]\label{thm: LS expansion}
     For $n \geq 4$, we have
     $$X_n = X_1^{-c_{n-1}}X_2^{-c_{n-2}} \sum_{\beta \in \calF(\calD_n)} X_1^{r|\beta|_1}X_2^{r\left(c_{n-1} - |\beta|_2\right)}$$
     and 
     $$X_{3-n} = X_2^{-c_{n-1}}X_1^{-c_{n-2}} \sum_{\beta \in \calF(\calD_n)} X_2^{r|\beta|_1}x_1^{r\left(c_{n-1} - |\beta|_2\right)}\,.$$
    \end{thm}

In order to show that our expansion formula \autoref{cor: LS modified expansion} is equivalent to \autoref{thm: LS expansion}, we define the map $\chi$ which connects the colored subpaths in both settings.

\begin{defn}\label{defn: chi}
We define a map $\chi: \calF(\calD_n) \to \calF'(\calD_n)$ that modifies the colored subpaths of $\beta \in \calF(\calD_n)$ via the following rules:
\begin{itemize}
    \item replace each red subpath in $\beta$ with the two subpaths obtained by splitting into its first edge and the remainder of the path, 
    \item change the color of $(m,w)$-green subpaths to become $(m,w)$-brown subpaths 
    \item the blue subpaths and single edges remain unchanged.
\end{itemize} 
\end{defn}

\begin{figure}
\centering
\includegraphics[width = 2.5in]{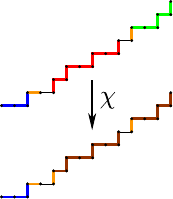}
\caption{The top image depicts a collection of colored Dyck subpaths in $\calF(\calD_6)$, namely, $\{\alpha(0,1),\alpha_4,\alpha(2,5),\alpha_{16},\alpha(6,8)\}$.  The bottom image depicts the corresponding collection of colored Dyck subpaths in $\calF'(\calD_{6})$, namely $\{\gamma(0,1),\alpha_4,\alpha_6,\gamma(2,5),\alpha_{16},\gamma(6,8)\}$, obtained by applying the map $\chi$ to the top collection.  }
\label{fig: colored path comparison}
\end{figure}

Note that the set of edges covered by $\beta$ is preserved under $\chi$.  \autoref{lem: simplified paths} establishes that $\chi$ is well-defined and is in fact a weight-preserving bijection with respect to $|\beta|_1$ and $|\beta|_2$.  The statement and proof of \autoref{lem: simplified paths} appear in \autoref{subsec: simplification proof}.

\subsection{\texorpdfstring{Vertices and Slopes in $\calD_n$}{Vertices and Slopes in Dn}}\label{subsec: vertices dyck}
We now prove several results relating to position of vertices in $\calD_n$ and the slopes of the line segments between northwest corners of $\calD_n$.  To help illuminate the recursive structure of the infinite path $\calD$, which contains each $\calD_n$ as a prefix, we define a map taking vertices to northwest corners in $\calD$.

\begin{defn}
Let $\mu:V(\calD) \to V(\calD)$ be the map sending $w_i$, the $i^\Th$ vertex of $\calD$, to  $v_i$, the $i^{\Th}$ northwest corner of $\calD$.
\end{defn} 

The following result describes the behavior of $\mu$ in terms of coordinates.

\begin{lem}\label{lem: corner map}
If $w_i \in V(\calD)$ has coordinates $(x,y)$, then $\mu(w_i) = \left((r-1)x + (r-2)y, x + y\right)$.
\end{lem}
\begin{proof}  
Fix $n$ large enough such that $(x,y) \in V(\calD_n)$.  For each $(x,y) \in V(\calD_n)$, the claim is equivalent to showing that the following inequalities hold:
$$\frac{x+y}{(r-1)x + (r-2)y} \leq \frac{c_{n-1}}{c_{n}-c_{n-1}} < \frac{x+ y+1}{(r-1)x + (r-2)y}\,.$$

In order to prove the first inequality, we first note that $$\frac{c_{n-2}}{c_{n-1}-c_{n-2}} \geq \frac{y}{x}\,,$$ which holds since $(x,y) \in V(\calD_n)$.  Cross multiplying and adding $yc_{n-2}$ to both sides yields
$$\frac{c_{n-2}}{c_{n-1}} \geq \frac{y}{x+y}\,.$$
Hence we have
$$\frac{c_{n-1}}{c_n-c_{n-1}} = \frac{c_{n-1}}{(r-1)c_{n-1}-c_{n-2}} = \frac{1}{(r-1)-\frac{c_{n-2}}{c_{n-1}}} \geq \frac{1}{(r-1)-\frac{y}{x+y}} = \frac{x+y}{(r-1)x+(r-2)y}\,,$$
as desired.

We can prove the second inequality similarly.  Since we have $\frac{c_{n-2}}{c_{n-1}-c_{n-2}} < \frac{y+1}{x},$
then cross multiplying and adding $(y+1)c_{n-2}$ to both sides yields
$$\frac{c_{n-2}}{c_{n-1}} < \frac{y+1}{x+y+1}\,.$$
Thus, we can conclude
\[\frac{c_{n-1}}{c_n-c_{n-1}} = \frac{1}{(r-1)-\frac{c_{n-2}}{c_{n-1}}} < \frac{1}{(r-1)-\frac{y+1}{x+y+1}} \leq  \frac{x+y+1}{(r-1)x+(r-2)y}\,.\qedhere\]
\end{proof}

We now study the slopes of the line segments between northwest corners of $\calD$, as these play a central role in the definition of colored Dyck subpaths in both our setting and the Lee--Schiffler setting.  We will utilize several classical results concerning maximal Dyck paths and Christoffel words, which can be found, for example, in \cite{BLRS}.    

\begin{defn}\label{defn: pi n}
For $n \geq 3$, we define the function $\pi_n: \{0,1,\dots,c_{n-1}\} \to \Z$ by
$$\pi_{n}(i) := xc_{n-2} - (i-x)(c_{n-1}-c_{n-2}) \text{ where } w_i = (x,i-x) \in \calD_n\,.$$
\end{defn}

\begin{rem}\label{rem: pi n}
Note that we have $\pi_n(0) = \pi_n(c_{n-1})= 0$.  It is a standard result from the theory of Christoffel words (see, for example, \cite[Lemma 1.3]{BLRS}) that the sequence $\pi_n(1),\pi_n(2),\dots,\pi_n(c_{n-1})$ is a permutation of the elements $\{0,1,\dots,c_{n-1}-1\}$, and this is order-isomorphic to sequence of distances from each vertex to the line segment between $w_0$ and $w_{c_{n-1}}$. Thus, $s(w_i,w_j) \geq s = s(w_0,w_{c_{n-1}})$ in $\calD_n$ if and only if $\pi_n(i) \leq \pi_n(j)$.
\end{rem}

\begin{exmp}
For $r = 3$, the values $\pi_6(0),\pi_6(1),\dots,\pi_6(21)$ are given by the following sequence:
$$0,8,16,3,11,19,6,14,1,9,17,4,12,20,7,15,2,10,18,5,13,0\,.$$
\end{exmp}

 We now prove several relations between the sequences $\{\pi_n(i)\}_{i=0}^{c_{n-1}}$ and $\{\pi_{n-1}(i)\}_{i=0}^{c_{n-2}}$.

\begin{lem}\label{lem: corner recursion}
For any vertex with coordinates $(x,y)$ in $\calD_{n-1}$ where $n \geq 4$, we have
$$\pi_{n}(rx + (r-1)y) = \pi_{n-1}(x + y)\,.$$
\end{lem}
\begin{proof}
By \autoref{lem: vertex Dn recursion} and \autoref{lem: corner map}, both $(x,y)$ and $\mu(w_{x+y}) = ((r-1)x+(r-2)y, x+y)$ are vertices of $\calD_n$.  Hence we can expand the right-hand side using \autoref{defn: pi n} and apply the relation $c_{n} = rc_{n-1} - c_{n-2}$ to obtain
\begin{align*}
\pi_{n}(rx + (r-1)y) &= ((r-1)x + (r-2)y)c_{n-2} - (x+y)(c_{n-1}-c_{n-2})\\
&= ((r-1)x + (r-2)y)c_{n-2} - (x+y)((r-1)c_{n-2} - c_{n-3})\\
&= -yc_{n-2} + (x+y)c_{n-3}\\
&= xc_{n-3} - y(c_{n-2} - c_{n-3})\;.
\end{align*}
Comparing with \autoref{defn: pi n}, we see that the final quantity is precisely $\pi_{n-1}(x + y)$, as desired.
\end{proof}

\begin{lem}\label{lem: nw corners}
For $n\geq 3$, the set $\left\{(x,y) \in V(\calD_n) : \pi_n(x+y) \in \{0,\dots,c_{n-1}-1\}\right\}$ is precisely the set of northwest corners of $\calD_n$.  
\end{lem}
\begin{proof}
We proceed inductively on $n$.  The base case $n = 3$ is straightforward to check.  \autoref{rem: pi n} implies that $\pi_n^{-1}(0) = \{v_0,v_{c_{n-1}}\}$ and that $\pi_{n}(i)$ takes each value of $\{1,2,\dots,c_{n-1}-1\}$ exactly once for $i \in \{2,3,\dots,c_{n-1}\}$.  For each vertex $(x,y)\in \calD_{n-1}$, we can evaluate
$$\pi_n(rx + (r-1)y) = \pi_{n-1}(x+y) \in \{0,1,2,\dots,c_{n-2}-1\}\,.$$ 
Moreover, the unique vertex $(x',y') \in \calD_n$ such that $x' + y' = rx + (r-1)y$ is $\mu(w_{x+y})$.  Since the image of $\mu$ applied to the set $V(\calD_{n-1})$ is the set of northwest corners of $\calD_n$, the desired set equality holds.
\end{proof}

\begin{cor}\label{lem: corner structure}
Applying the morphism $\lambda = \{E \mapsto E^{r-1}N, N \mapsto E^{r-2}N\}$ to the Christoffel word for $\calD_n$ yields the Christoffel word for $\calD_{n+1}$.

Applying the morphism $\theta \circ \lambda \circ \theta^{-1}$  to the Christoffel word for $\calC_n$ yields the Christoffel word for $\calC_{n+1}$.
\end{cor}
\begin{proof}
Let $w_i$ denote the Christoffel word for $\calD_i$.  Applying the map $\mu$ to $V(\calD_n)$ takes a vertex $(x,y)$ to $\left((r-1)x + (r-2)y, x+y\right)$.  Hence by \autoref{lem: nw corners}, the index $j$ of the $k^\Th$ vertical edge $\alpha_j$ in $\calD_{n+1}$, i.e., the positions of $y$'s in $w_{n+1}$, is equal to $$r|\{\alpha_i^n : i < j, \alpha_i^n \text{ is horizontal}\}| + (r-1)|\{\alpha_i^n : i < j, \alpha_i^n \text{ is vertical}\}| + 1\,.$$
This is precisely the position of $y$'s in the word $\lambda(w_n)$, hence we can conclude that $\lambda(w_n) = w_{n+1}$.  The second statement follows from \autoref{rem: C to D}.
\end{proof}

  Finally, we can determine the rest of the values of $\pi_{n}$ from its values on the northwest corners.  

\begin{obs}\label{obs: corner extension}
Suppose that $(i,j) \in \calD_n$ is not a northwest corner.  Let $(i',j)$ be the corner vertex immediately preceding $(i,j)$.  Then it follows from the definition of $\pi_n$ that
$$\pi_n(i + j) = \pi_n(i' + j) + (i-i')c_{n-2}\;.$$
\end{obs}

\subsection{Proof of the Simplification}\label{subsec: simplification proof}

We now show that the map $\chi$ (see \autoref{defn: chi}) is well-defined and weight-preserving with respect to $|\beta|_1$ and $|\beta|_2$.  This shows that our definition of $\calF'(\calD_n)$ is, in a sense, equivalent to that of Lee-Schiffler.  

\begin{obs}\label{obs: label sets}
\autoref{lem: corner recursion} states that the values of $\pi_n$ on the northwest corners of $\calD_n$ lie in the set $\{0,1,\dots,c_{n-2}-1\}$.  Since $\pi_n$ is injective on the interior of $\calD_n$, this implies that the values of $\pi_n$ of the vertices which are not northwest corners constitute the set $\{c_{n-2},c_{n-2}+1,\dots,c_{n-1}-1\}$.
\end{obs}

The colors of paths in Lee and Schiffler's setting depends on the number of northwest corners that one needs to traverse from the starting endpoint of the path until the slope to the northwest corner was at least the slope of the diagonal.  We also consider the slopes between vertices on $\calD_n$ that are not necessarily northwest corners, which will help us to derive results about the northwest corners.  We in turn use this to determine which vertex $w_{d(i)}$ to the east of a given vertex $w_i$ is the first such that the slope between $w_i$ and $w_{d(i)}$ is at least that of the diagonal.

\begin{defn}\label{defn: d i}
For $0 \leq i < c_{n-1}$, we define
$$d(i) = \min\left(\{j \in \{i+1,i+2,\dots,c_{n-1}\}: s(w_i,w_j) \geq s(0,c_{n-1})\}\right)\,.$$  
\end{defn}

Note that $d(i)$ is well-defined since $s(i,c_{n-1}) \geq s(0,c_{n-1})$ for all $0 \leq i < c_{n-1}$.  We now show some properties of how the functions $d$ and $\mu$ interact.

\begin{cor}\label{cor: vertex form step}
Suppose $w_i$ is a northwest corner of $\calD_n$ for $n \geq 4$. Let $w_i = \mu\left(w_{i'}\right)$ for some $w_{i'} \in \calD_{n-1}$, and let $(x,y) = w_{d(i')}$ in $\calD_{n-1}$ with $d(i')-i' = mc_{n-2} - wc_{n-3}$.  Then we have $w_{d(i)} = ((r-1)x+(r-2)y, x+y)$, and $d(i)-i = mc_{n-1}-wc_{n-2}$.    

If $w_i$ is not a northwest corner, then $w_{d(i)}$ is the northwest corner immediately following $w_i$.   In particular, $d(i) - i \leq r-1$.  
\end{cor}
\begin{proof}
By \autoref{obs: label sets}, $w_{d(i)}$ must be a northwest corner.  Thus \autoref{lem: nw corners} implies that $w_{d(i)}$ must be the image of $w_{d(i')}$ under the correspondence between vertices of $\calD_{n-1}$ to northwest corners of $\calD_{n}$.  It is straightforward to check that the distances follow the formula described under this correspondence.  The second claim follows directly from \autoref{obs: label sets}.
\end{proof}

From this, we can determine that the values $d(i)-i$ are of a particular form.  

\begin{lem}\label{lem: mw corners}
For all positive integers $i$, we have $d(i) - i = c_m - wc_{m-1}$ for a unique choice of $2 \leq w \leq r-1$ and $3 \leq m \leq n-2$. 
\end{lem}
\begin{proof}
We prove this via induction on $n$.  The base case $n=3$ is straightforward to check.  Suppose the statement holds on $\calC_{n-1}$.  Whenever $w_i$ is not a northwest corner of $\calC_n$, then we have $d(i) - i < r$ by \autoref{cor: vertex form step}, so $d(i) - i = c_3 - wc_2$ for some $2 \leq w \leq r-1$ or $d(i) = r-1 = c_4 - (r-1)c_3$.  Otherwise, if $w_i$ is a northwest corner of $\calC_n$, then $w_i = \mu(w_{i'})$ for some $w_{i'} \in \calC_{n-1}$.  By assumption, $d(i')-i' = c_m - wc_{m-1}$ for some appropriate choice of $m,w$.  Applying \autoref{cor: vertex form step}, we can conclude that $d(i)-i = c_{m+1} - wc_{m}$.  

The fact that this representation is unique follows immediately from the fact that, for all $m \geq 2$, we have
$c_m - (r-1)c_{m-1} > c_{m-1} - 2c_{m-2}$.
\end{proof}

We are now ready to establish that the value $t(i) - i$, as given in \autoref{thm: path slope result}, can also be represented as $c_m - wc_{m-1}$ for an appropriate choice of $m,w$.

\begin{proof}[Proof of \autoref{thm: path slope result}]
We prove this via induction on $n$, with the straightforward base case $n=3$.  Fix a northwest corner $v_i = w_j \in \calC_n$, and let $t(i)$ be the minimum positive integer such that $s(v_i,v_{t(i)})\geq s$.  Since $v_i$ is a northwest corner, we have $v_i = \mu(w_{i'})$ for some $w_{i'} \in \calC_{n-1}$.   Applying \autoref{lem: mw corners}, we have that $d(i') - i' = c_m - wc_{m-1}$ for a unique choice of $2 \leq w \leq r-1$ and $3 \leq m \leq n-2$.  By \autoref{rem: pi n}, we have that $t(i) - i = d(i') - i'$.  Thus $t(i) - i$ is of the desired form.
\end{proof}

The following lemma establishes that $\chi$ preserves weights.

\begin{lem}\label{lem: simplified paths}
For all $n \geq 3$, we have
$\beta \in \calF(\calD_n)$ if and only if $\chi(\beta) \in \calF'(\calD_n)$.  Moreover, we have $|\beta|_i = |\chi(\beta)|_i$ for $i = 1,2$.
\end{lem}
\begin{proof}[Proof of \autoref{lem: simplified paths}]
Fix $\beta \in \calF(\calD_n)$.  Suppose $\alpha(i,k) \in \beta$.  By \autoref{thm: path slope result}, either we have that $\alpha(i,k)$ is $(m,w)$-green, or we have that $t(i) - i = 1$.  In this case, condition $(2^*)$ of \autoref{def: brown path} enforces that the edge immediately preceding $\alpha(i,k)$ is contained in $\beta_j$.  By the non-overlapping condition for membership in $\calF(\calD_n)$, we have $\beta_j \neq \alpha(i',k')$ for any $i',k'$.  Thus $\beta_j = \alpha_i'$ for some $i'$.  In particular, it does not contribute to $|\beta|_1$ and contributes $1$ to $|\beta|_2$, which is the same as if we had considered $\alpha(i,k)$ to contain this preceding edge. 
\end{proof}

We can now combine the results about the map $\chi$ in order prove the expansion formula in our setting.

\begin{proof}[Proof of \autoref{cor: LS modified expansion}]
This modified expansion formula follows immediately from \autoref{thm: path slope result} and the expansion formula (\autoref{thm: LS expansion}) given by Lee and Schiffler \cite{LS}.
\end{proof}

We now furthermore discuss a generalization of \autoref{cor: LS modified expansion} to the setting of the \emph{framed $r$-Kronecker cluster algebra with principal coefficients}.  While there is a more general theory of cluster algebras with coefficients (see, for example, \cite{FZ4}), we will give a brief explicit description of this cluster algebra here.  For a positive integer $r$, initial cluster variables $X_1,X_2$ and coefficient variables $Y_1,Y_2$, we consider the sequence $\{\widetilde{X}_n\}_{n \in \Z}$ and $\{\widetilde{Y}_n\}_{n \in \Z}$ defined by

\begin{align*}
\widetilde{Y}_{n+1} &= \frac{\widetilde{Y}_n^r}{\widetilde{Y}_{n-1}}\,, \text{ where } \widetilde{Y}_1 = Y_1 \text{ and } \widetilde{Y}_2 = Y_1^rY_2,\;\\
\widetilde{X}_{n+1} &= \frac{\widetilde{X}_n^r + \widetilde{Y}_{n-1}}{\widetilde{X}_{n-1}}\,,\text{ where } \widetilde{X}_1 = X_1 \text{ and } \widetilde{X}_2 = X_2\,.
\end{align*}

The use of tildes is to distinguish the settings with and without coefficients.  Let $\mathbb{P}$ be the tropical semifield $\Trop[Y_1,Y_2]$.  The \emph{framed $r$-Kronecker cluster algebra $\widetilde\calA(r,r)$ with principal coefficients} is the $\Z\mathbb{P}[\widetilde{X}_1,\widetilde{X}_2]$ algebra generated by the cluster variables $\{\widetilde{X}_n\}_{n \in \Z}$.  Note that when we specialize to the case $Y_1 = Y_2 = 1$, we recover the classical $r$-Kronecker cluster algebra.

\begin{cor}\label{cor: LS modified expansion coeffs}
 Consider the framed $r$-Kronecker cluster algebra $\widetilde\calA(r,r)$ with principal coefficients, having initial cluster variables $X_1,X_2$ and coefficient variables $Y_1,Y_2$.  For $n \geq 4$, the cluster variable $\widetilde{X}_n$ is given by
 $$\widetilde{X}_n = X_1^{-c_{n-1}}X_2^{-c_{n-2}} \sum_{\beta \in \calF'(\calD_n)} X_1^{r|\beta|_1}X_2^{r\left(c_{n-1} - |\beta|_2\right)}Y_1^{|\beta|_2}Y_2^{|\beta|_1}$$
 and 
 $$X_{3-n} = X_2^{-c_{n-1}}X_1^{-c_{n-2}} \sum_{\beta \in \calF'(\calD_n)} X_2^{r|\beta|_1}X_1^{r\left(c_{n-1} - |\beta|_2\right)}Y_2^{-|\beta|_2}Y_1^{-|\beta|_1}\,.$$
\end{cor}
\begin{proof}
 Setting $\deg(X_i) = e_i$ and $\deg(Y_i) = (-1)^{i+1} re_{3-i}$, it is known that the cluster variable $\widetilde{X}_n$ is a homogeneous Laurent polynomial by \cite[Proposition 6.1]{FZ4}.  Moreover, this degree is readily calculated from the recurrence relations on \emph{g-vectors} to be $-c_{n-1}e_1 + c_ne_2$ for $n \geq 2$ and $c_{-n + 3}e_1-c_{-n+2}e_2$ for $n < 2$ (see, for example, \cite[Subsection 4.1]{Lin}).  This determines the powers of $Y_1$ and $Y_2$ that must appear in each monomial term, yielding the above expansion formula directly from \autoref{cor: LS modified expansion}.
\end{proof}

\section{Bijection between Compatible Pairs and Colored Subpaths of Dyck Paths}\label{sec: bijection}
In this section, we prove a conjecture of Feiyang Lin that the map $\Phi$, constructed by Lin and described in \autoref{defn: Phi}, is a bijection between the collections of colored subpaths introduced by Lee-Schiffler \cite{LS} and the compatible pairs introduced by Lee-Li-Zelevinsky \cite{LLZ}.  This shows the correspondence between the objects summed over by each set of authors in their expansion formulas for rank-$2$ cluster algebras.  Specifically, we show that Lin's map is a bijection between collections $\beta$ of colored Dyck subpaths in $\calF'(\calD_n)$ with a fixed $|\beta|_1$ and $|\beta|_2$ and compatible pairs on $\calC_n$ consisting of $|\beta|_1$ vertical edges and $(c_{n-1}-|\beta|_2)$ horizontal edges.

\subsection{Compatible Pairs and Lin's Map}
The rank-2 cluster expansion formula given by Lee-Li-Zelevinsky \cite{LLZ} sums over certain subsets of edges of a maximal Dyck path, known as compatible pairs, which we discuss here. We will mainly work over $\calC_n$, though sometimes we work in more generality.  Let $S_1$ be a subset of the vertical edges of a maximal Dyck path $\calP(a,b)$, and let $S_2$ be a subset of the horizontal edges of $\calP(a,b)$.

In order to study compatible pairs, Li, Lee, and Zelevinsky \cite{LLZ} introduced the notion of the ``shadow'' of a set of edges.  While they only defined shadows of subsets of vertical edges, we extend this notion to subsets of horizontal edges as well.  These notions will be used throughout our construction of the bijection between collections of colored Dyck subpaths and compatible pairs.

\begin{defn}\label{def: compatible pair}
For a vertical edge $\nu \in S_2$ with upper endpoint $w$, we define its \emph{local vertical shadow}, denoted $\sh(\nu;S_2)$, to be the set of horizontal edges in the shortest subpath $\overrightarrow{tw}$ of $\calP(a,b)$ such that $|tw|_1 = r|\overrightarrow{tw} \cap S_2|_2$.  Analogously, for a horizontal edge $\eta \in S_1$ with left endpoint $u$, we define its \emph{local horizontal shadow}, denoted $\sh(\eta,S_2)$, to be the set of vertical edges in the shortest subpath $\overrightarrow{ut}$ of $\calP(a,b)$ such that $|ut|_2 = r|\overrightarrow{ut} \cap S_1|_1$.  If there is no such subpath $\overrightarrow{tw}$ or $\overrightarrow{ut}$, respectively, then we define the local vertical (resp., horizontal) shadow to be the entire set of horizontal (resp., vertical) edges in $\calP(a,b)$.

 For $S \subseteq S_i$ where $i \in \{1,2\}$, let $\displaystyle \sh(S;S_i) = \bigcup_{\alpha \in S} \sh(\alpha;S_i)$, and write $\sh(S_i) := \sh(S_i;S_i)$.  
\end{defn}

\begin{obs}\label{obs: shadow length}
It is a straightforward consequence of \autoref{def: compatible pair} that for $S \subseteq S_1$, we have $|\sh(S)| = \min(b,\, r|S|)$.  Similarly, for $S \subseteq S_2$, we have $|\sh(S)| = \min(a,\, r|S|)$
\end{obs}

The expansion formula for cluster variables given by Lee, Li, and Zelevinsky has monomials corresponding to compatible pairs on $\calC_n$.  Their expansion formula works in the more general setting of elements of the greedy basis, which contains the cluster variables.  For further details on the greedy basis, see \cite{LLZ}.  We present their formula in the special case of cluster variables.

\begin{thm}{\cite[Theorem 1.11]{LLZ}}\label{thm: LLZ expansion}
For each $n \geq 1$, the cluster variable $X_n$ in $\calA(r,r)$ is given by
$$X_n = x_1^{-c_{n-1}}x_2^{-c_{n-2}} \sum_{(S_1,S_2)}x_1^{r|S_2|}x_2^{r|S_1|}\,,$$
where the sum is over all compatible pairs $(S_1,S_2)$ in $\calC_n$.
\end{thm}

\begin{figure}
\centering
\includegraphics[width = 4in]{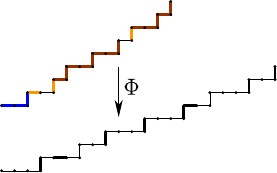}
\caption{The top image depicts a collection of colored Dyck subpaths in $\calF'(\calD_6)$, the same as in \autoref{fig: colored path comparison}.  The bottom image depicts the corresponding compatible pair $(S_1,S_2)$ in $\calC_{6}$, obtained by applying the map $\Phi$ to the collection of colored Dyck subpaths.  The edges that are contained in either $S_1$ or $S_2$ are depicted by bold edges in the lower image.  In particular, we have $S_1 = \{\eta_5,\eta_{15}\}$ and $S_2 = \{\nu_1,\nu_3,\nu_4,\nu_5,\nu_7,\nu_8\}$.  Thus $\wt((S_1,S_2)) = X_1^{18}X_2^{6}$.}
\label{fig: Phi map}
\end{figure}

Lin's map from collections colored Dyck subpaths in $\calD_n$ to compatible pairs in $\calC_n$ is described in \autoref{defn: Phi}.  An example is shown in \autoref{fig: Phi map}.  Lin conjectured that the map $\Phi$ is a bijection between the desired sets \cite[Conjecture 3]{Lin}, which essentially involves showing that $\Phi(\beta)$ is indeed a compatible pair for every $\beta \in \calF'(\calD_n)$.  We verify this claim in the next subsection.  Lin made partial progress toward this conjecture by showing that it was sufficient to consider only compatible pairs satisfying a certain irreducibility condition \cite[Proposition 4.8.4, Conjecture 4]{Lin}.  We proceed by a different approach than Lin, so our methods do not rely on this simplification.

In order to show the correspondence between the Lee-Schiffler and Lee-Li-Zelevinsky expansion formulas, one needs to show not only bijectivity between the sets summed over, but also that the resulting monomials correspond.  Lin defined a weight function for collections of colored subpaths and for compatible pairs that keeps track of the exponents associated to these monomials.

\begin{defn}\label{def: phi weights}
We define the \emph{weight} of a compatible pair $(S_1,S_2)$ by
$$\wt((S_1,S_2)) = X_1^{r|S_2|}X_2^{r|S_1|}$$
and the \emph{weight} of a collection of colored subpaths $\beta \in \calF'(\calD_n)$ by
$$\wt_n(\beta) = X_1^{r|\beta|_1}X_2^{r(c_{n-2}-|\beta|_2)}\,.$$
\end{defn}

Lin showed that $\Phi$ is a weight-preserving map from a superset of $\calF(\calD_n)$ to $\calF(\calC_n)$, and conjectured that it restricted to a bijection between $\calF(\calD_n)$ and $\calF(\calC_n)$.  We prove this in the next subsection.  We convert Lin's map into the setting of $\calF'(\calD_n)$ instead of $\calF(\calD_n)$ in order to make easier reference to the results of the previous section, though it is straightforward to show the equivalence between these two settings.


\subsection{Proof of Bijectivity}
We now proceed to show that Lin's map $\Phi$ indeed takes every collection of colored subpaths in $\calF'(\calD_n)$ to a unique compatible pair on $\calC_n$.  It then follows from the work of Lee-Schiffler \cite{LS} and Lee-Li-Zelevinsky \cite{LLZ} that $\Phi$ is a bijection.  For $2 \leq w \leq r-1$ and $m \geq 3$, we define $a_{m,w}$ to be the quantity $c_{m}-wc_{m-1}$. We can use the quantities $a_{m,w}$ to define the size of images of atomic colored paths under $\Phi$, as well as their shadows. 

\begin{obs}\label{obs: awm identity}
It is readily deduced from the recursive definition of the sequence $c_n$ that for $w,m \geq 1$, we have $ra_{m,w} = a_{m+1,w} + a_{m-1,w}$.
\end{obs}

In order to establish that the conditions for compatibility correspond to the conditions for membership in $\calF'(\calD_n)$ via $\Phi$, we first show that this is true for certain simple colored subpaths.

\begin{defn}
We call a subpath of $\calD_n$ \emph{atomic} if it consists of a single edge, is blue, or is an $(m,w)$-brown path of the form $\gamma(i,i+a_{m,w})$.
\end{defn}

\begin{obs}
Any subpath of $\calD_n$ can be written uniquely as a union of atomic components meeting only at vertices such that 
\begin{enumerate}[(i)]
    \item no component is a single edge unless the entire path is a single edge, and
    \item only the last component can be blue.
\end{enumerate}
This decomposition is obtained by removing minimal brown paths from the front until only a blue path remains.
\end{obs}

Note that when we decompose a path into its atomic components, adjacent components will necessarily overlap at a vertex.  Thus after this decomposition, the set of paths may no longer be non-overlapping, and hence not in $\calF'(\calD_n)$.  An example is shown in \autoref{fig: atomic decomp}.

\begin{figure}
\centering
\includegraphics[width = 2.5in]{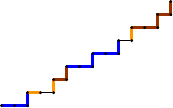}
\caption{The decomposition of the collection $\beta$ of subpaths from \autoref{fig: Phi map} into its atomic components.  Note that $(3,2)$-brown path $\gamma(2,3)$ and blue path $\gamma(3,5)$, which meet at the vertex $v_3$, are the two atomic components of the brown path $\gamma(2,5)$ from $\beta$.}
\label{fig: atomic decomp}
\end{figure}

We now study the structure of the $(m,w)$-brown paths as Christoffel words.  Recall the morphism $\lambda = \{E \mapsto E^{r-1}N, N \mapsto E^{r-2}N\}$.

\begin{lem}\label{lem: brown Christoffel structure}
The Christoffel word for an atomic $(m,w)$-brown path is $\lambda^{(m-2)}(E^{r-w-1} N)$, where $\lambda^0$ is the identity map, which has length $a_{m+1,w}$.
\end{lem}
\begin{proof}
Let $\gamma(i,k)$ be an atomic $(m,w)$-brown path, and let $\rho$ denote the corresponding Christoffel word.  Fix $i', k' \in \Z$ such that $v_i = \mu(w_{i'})$ and $v_k = \mu(w_{k'})$.  

First note that if $w_{i'}$ is not a northwest corner, then by \autoref{obs: label sets} we have $k'-i' \leq  r$.  Thus we have $m = 3$ and  $\rho = \lambda(E^{r-w-1} N)$ where $r-w-1 = k' - i' - 1$, which is of the desired form.

Now suppose that $w_{i'} = v_{i''}$ is a northwest corner.  We automatically have that $w_{k'} = v_{k''}$ is a northwest corner from \autoref{obs: label sets}.  We aim to show that $\gamma(i'',k'')$ is an $(m-1,w)$-brown path.  Thus the statement would follow from induction, since the Christoffel word corresponding to $\gamma(i,k)$ is given by applying $\lambda$ to the word corresponding to $\gamma(i'',k'')$.   

In order to show that $\gamma(i'',k'')$ is an $(m-1,w)$-brown path, we study the slopes from $s(w_{i'},w_{i' + j})$ for $1 \leq j \leq k'-i' = k-i$.  It readily follows from the recurrence for the sequence $c_n$ and the formula for $\mu$ given in \autoref{lem: corner map} that 
$$s(w_{i'},w_{i'+ j}) - s(v_0,v_{c_{n}}) = s(\mu(w_{i'}),\mu(w_{i'+ j})) - s(v_0,v_{c_{n+1}}) = s(v_i,v_{i+j}) - s(v_0,v_{c_{n+1}})\,.$$
Since $s(v_i,v_{i+j}) - s(v_0,v_{c_{n+1}}) < 0$ for $1 \leq j \leq k-i$ and $s(v_i,v_k) \geq s(v_0,v_{c_{n+1}})$, the same holds for $s(w_{i'},w_{i'+ j}) - s(v_0,v_{c_{n}})$.  That is, the slope from $w_{i'}$ to any vertex on $\gamma(i'',k'')$ does not exceed that of the diagonal, except the slope from $w_{i'}$ to $w_{k'}$.  

We now just need to determine $k'' - i''$, or equivalently, the number of vertical edges in $\gamma(i'',k'')$.  Since $\gamma(i,k)$ as $a_{w,m}$ vertical edges, then we know $\gamma(i'',k'')$ has $a_{w,m}$ total edges.  Since the $a_{w,m}$ uniquely determine $w,m$, we can conclude from the inductive hypothesis that $\gamma(i'',k'')$ is $(m-1,w)$-brown.  Moreover, by \autoref{obs: awm identity}, we have
$$(r-1)a_{w-1,m}+r(a_{w,m}-a_{w-1,m}) = a_{w+1,m}\,,$$
so we can conclude that $\gamma(i,k)$ has length $a_{w+1,m}$.
\end{proof}

We are interested in the portion of the path spanned by the vertical shadow of the image of an $(m,w)$-brown path.  We determine the structure of this portion of the path with the following result.

\begin{cor}\label{cor: shadow shape}
The $a_{m-1,w}$ edges preceding an $(m,w)$-brown path form an $(m-2,w)$-brown path or, when $a_{m-1,w} < r$, a path corresponding to the Christoffel word $E^{a_{m-1,w} - 1}N$.
\end{cor}
\begin{proof}
By definition, the edge preceding an $(m,w)$-brown path is vertical, so the latter statement follows immediately.

We prove the first claim via induction on $m$.  For $m \leq 5$, we note that $a_{m-1,w} \leq r-1$, and hence the preceding path is of the form $E^{a_{m-1,w}-1}N$.  Let $\rho$ denote the Christoffel word formed by the $a_{m-2,w}$ edges preceding an $(m-1,w)$-brown path.  Then by \autoref{lem: brown Christoffel structure}, we can obtain the word corresponding to the $a_{m-1,w}$ edges preceding an $(m,w)$-brown path by applying $\lambda$ to $\rho$.  By the inductive hypothesis and \autoref{lem: brown Christoffel structure}, we can conclude that $\lambda(\rho)$ is an $(m-2,w)$-brown path.
\end{proof}

\begin{cor}\label{cor: shadow size}
Let $\gamma(i,k)$ be an $(m,w)$-brown path in $\calD_n$, and let $(S_1,S_2) = \Phi\left(\{\gamma(i,k)\}\right)$. 
Then the shadow of $S_2$ has length $a_{m+1,w} + a_{m-1,w}$.
\end{cor}
\begin{proof}
By \autoref{lem: brown Christoffel structure}, it follows that $S_2$ consists of $a_{m,w}$ consecutive vertical edges.  Thus, by \autoref{obs: shadow length}, the shadow will contain $\min(ra_{m,w},c_{n-1})$ horizontal edges.  Applying \autoref{obs: awm identity}, we see that $ra_{m,w}= a_{m+1,w} + a_{m-1,w}$.  We then have by the bounds on $w$ and $m$ that
$$ra_{m,w} = a_{m+1,w} + a_{m-1,w} \leq (c_{m+1} - 2c_{m}) + c_{m-1} \leq c_{m+1} \leq c_{n-1}\,.$$
So we can conclude the length of the shadow is $a_{m+1,w} + a_{m-1,w}$.
\end{proof}

We can now establish that the image under $\Phi$ of an atomic $(m,w)$-brown path is a compatible pair in $\calC_n$.  As we will later see, this encodes most of the complexity of the compatibility conditions on $\calC_n$.

\begin{lem}\label{lem: brown compatibility}
Let $\gamma(i,k)$ be an atomic $(m,w)$-brown path in $\calD_n$ and $\gamma_j$ be one of the $a_{m-1,w}$ edges preceding $v_i$.  Then $\Phi(\{\gamma(i,k),\gamma_j\})$ is a compatible pair.
\end{lem}
\begin{proof}
Let $S_2 = \Phi_2(\{\gamma(i,k)\})$.  Let $\rho'$ denote path formed by the $a_{m-1,w}$ edges preceding $v_i$, and let $\rho = \sigma(\rho')$.  It follows from \autoref{cor: shadow size} and \autoref{cor: shadow shape} the shadow of $\Phi_2$ spans a path of type $\rho \lambda^2(\rho)$.  

Let $\upsilon$ be the path composed of $a_{m,w}$ north steps and $a_{m-1,w}$ west steps starting from the vertex immediately below $\Phi(v_i)$, defined as follows: for $i \geq 2$, the $i$-th north step of $\upsilon$ is $(i-1)r$ units to the west of the $(i-1)$-st edge in $S_2$.  By definition, $\upsilon$ forms the eastern border of the shadow of each edge of $S_2$ except the last.  Moreover, the position of west steps in $\upsilon$ is determined by the occurrence of subwords $E^{r-1}N$ and $E^rN$ in $\lambda^2(\rho)$.  From the definition of $\lambda$, it follows that Christoffel word corresponding to the $90$ degree clockwise rotation of $\upsilon$ is precisely $\lambda^{-1}(\lambda^2(\rho)) = \lambda(\rho)$. 

\begin{figure}
    \centering
    \includegraphics[width = 6in]{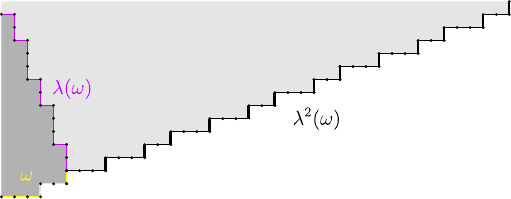}
    \caption{The yellow and black path constitute a portion of $\calD$.  The thick black edges depict those that are contained in $\Phi_2(\gamma)$, where $\gamma$ is a $(6,2)$-brown path.  The shadow of $\Phi_2(\gamma)$ spans the black and yellow paths.  The purple path shows the left endpoint of the shadow of the vertical edge to its right. Note that the yellow and purple paths overlap in one edge.}
    \label{fig: brown shadow}
\end{figure}
We are now interested in the vertical distance from each horizontal edge $\eta_i$ of $\rho$ to that of $\lambda(\rho)$.  This determines the maximum number $M$ of horizontal edges to the right of and including $\eta_i$ that can be included in $S_1$ while satisfying the compatibility conditions.  Namely, this distance is the maximum height of the shadow at $\eta_i$, i.e., $rM$.  Using the same inductive techniques as in \autoref{lem: brown Christoffel structure}, it is readily seen from the base case $\rho = E^wN$ that this distance is $r(a_{m-1,w} - i + 1)$ when $i > 1$.  Hence it is possible that any combination of these edges is contained in $S_1$.  When $i = 1$, the distance is $ra_{m-1,w} - 1$.  Since this is less than $ra_{m-1,w}$ but greater than $r(a_{m-1,w}-1)$, it is not possible that all edges of $\rho$ are contained in $S_1$, but it is possible that all but one are. That is, the pair $(S_1,S_2)$ is compatible if and only if at least one edge of $\rho$ is not contained in $S_1$.  
\end{proof}

\begin{exmp}
\autoref{fig: brown shadow} illustrates the construction in the proof of \autoref{lem: brown compatibility}.  Observe that the purple path is a $90$ degree counterclockwise rotation of a Dyck path.  Letting $\rho = E^2NEN$ denote the Christoffel word for the yellow path, observe that the Christoffel word corresponding to the (90 degree clockwise rotation of the) purple path is given by $\lambda(\rho)$.  Moreover, the Christoffel word corresponding to the black path is given by $\lambda^2(\rho)$.   The vertical distance between the purple and yellow paths is precisely the maximum height of the shadow at each horizontal yellow edge such that the compatibility conditions are satisfied.
\end{exmp}

\begin{lem}\label{lem: atomic single}
Let $\beta$ consist of a single atomic path and, if $\beta$ is $(m,w)$-brown, one of the $a_{m-1,w}$ edges preceding this path.  Then $\Phi(\beta)$ is a compatible pair.
\end{lem}
\begin{proof}
We break into three cases based on the form of $\beta$.
If $\beta$ consists of a single edge, then $\Phi(\beta)$ has no vertical edges and hence is compatible.  If $\beta$ is $(m,w)$-brown, then this is precisely the result of \autoref{lem: brown compatibility}.  

Thus, the only remaining case is when $\gamma(i,k)$ is blue.  Then it is the prefix of an atomic $(m,w)$-brown path, obtained by extending $\gamma(i,k)$ until its slope exceeds that of the diagonal.  Let $\gamma(i,k')$ denote this atomic $(m,w)$-brown path, and let $\beta' = \{\gamma(i,k'),\eta_j\}$ where $\eta_j$ is the $(a_{m-1})$-th edge preceding $v_i$.  In the previous case, we have shown that $\Phi(\beta')$ is compatible.  Note now that $\Phi_2(\beta) \subseteq \Phi_2(\beta')$ and $\sh(\Phi_2(\beta)) \subseteq \Phi_1(\beta')$.  Therefore the compatibility of $\Phi(\beta)$ follows directly from the compatibility of $\Phi(\beta')$.
\end{proof}

Now that we have handled the case of atomic paths, we show that we can determine the compatibility of the image of many colored subpaths paths by restricting to the compatibility conditions on each of its atomic components.

\begin{defn}
Given a pair $(S_1,S_2)$ on $\calP(a_1,a_2)$, and let $(S_1',S_2')$ be a pair on $\calP(a_1',a_2')$.  We define the \emph{insertion of $(S_1',S_2')$ into $(S_1,S_2)$ at position $(j_1,j_2)$} to be the pair $(S_1'', S_2'')$ on $\calP(a_1+a_1',a_2+a_2')$ determined as follows:
\[
e_i \in S_k'' \iff \begin{cases}
e_i \in S_k \text{ and } 1 \leq i \leq j_k\,, \text{or}\\
e_{i - j_k} \in S_k' \text{ and } j_k < i \leq j_k + a_k'\,,\text{or}\\
e_{i - a_k'} \in S_k \text{ and } j_k + a_k' < i \leq a_k + a_k'\,,
\end{cases}
\]
for $k \in \{1,2\}$ and $(j_1,j_2) \in V(\calP(a_1,b_1))$.  Here, each $e_i$ refers to the $i^\Th$ horizontal or $i^\Th$ vertical edge of the corresponding path, where the orientation of the edge is determined by the context.
\end{defn}

\begin{lem}\label{lem: pair insertion}
Let $(S_1,S_2)$ and $(S_1',S_2')$ be compatible pairs on $\calP(a_1,a_2)$ and $\calP(a_1',a_2')$, respectively, and suppose that $(S_1',S_2')$ has non-spanning shadows.  Then for any $(j_1,j_2) \in V(\calP(a_1,a_2))$, the insertion $(S_1'',S_2'')$ of $(S_1',S_2')$ into $(S_1,S_2)$ at position $(j_1,j_2)$ is a compatible pair on $\calP(a_1+a_1',a_2+a_2')$.
\end{lem}
\begin{proof}
Since $(S_1',S_2')$ is a compatible pair on $\calP(a_1',a_2')$, then either $rS_1' < a_2'$ or $rS_2' < a_1'$.  Without loss of generality, suppose $rS_1' \leq a_2'$.  

We can then see that for $e_i \in S_1''$ with $1 \leq i \leq j_1$, we have 
$$|\sh(e_i;S_1'')|\leq |\sh(e_i;S_1)| + \sh(e_1,S_1'')| < |\sh(e_i;S_1)| + a_2'\,.$$
Similarly, for $j_2 + a_2' < i \leq a_2 + a_2'$, we have
$$|\sh(e_i;S_2'')| \leq |\sh(e_{i-a_2'};S_2)| + |\sh(e_{a_2'},S_2'')| \leq |\sh(e_i;S_2)| + a_1'\,.$$
The lengths of the shadows at the other edges, with the corresponding shift in indices, is the same as in the original paths.  Thus the horizontal and vertical shadows will never intersect, so the pair is indeed compatible.
\end{proof}

This allows us to combine our results on atomic paths in order to handle any collection of subpaths in $\calF'(\calD_n)$

\begin{thm}\label{thm: bij forward}
If $\beta \in \calF'(\calD_n)$, then $\Phi(\beta)$ is a compatible pair.
\end{thm}
\begin{proof}
Let $(S_1,S_2) = \Phi(\beta)$
We proceed by induction on the number of atomic components in $\beta$, which we denote by $t$.  Note that if $t=0$, then $S_2$ is empty and so $(S_1,S_2)$ is compatible.

If $\beta$ has one atomic component, the compatibility follows directly from \autoref{lem: atomic single}.

If we add an additional singular edge to $\beta$, then the resulting pair is a subset of the original, and hence compatible.  Otherwise, suppose we add an atomic component to $\beta$ that appears to the left of all other atomic components.  Then we can view the resulting path as the insertion of the atomic path into the previous compatible pair.  Since both paths involved in the insertion are compatible, we can conclude using \autoref{lem: pair insertion} that $\Phi(\beta)$ is compatible.
\end{proof}

\begin{lem}\label{lem: bij inverse}
Every compatible pair in $\calC_n$ is the image of some $\beta \in \calF'(\calD_n)$.
\end{lem}
\begin{proof}
We know from \autoref{thm: bij forward} that $\Phi(\beta)$ for $\beta \in \calF'(\calD_n)$ is a compatible pair in $\calC_n$.  Moreover, we know from \autoref{thm: LS expansion} and \autoref{thm: LLZ expansion} that $\calF'(\calD_n)$ and the set of compatible pairs in $\calC_n$ are equinumerous, since both are the result of the substitution $X_1 = X_2 = 1$.  Lastly, we know from the work of Lin \cite[Proposition 4.7.3]{Lin} that $\Phi$ is injective.  Thus $\Phi$ is also surjective.
\end{proof}

Combining the results proven above along with work of Lee-Schiffler and Lee-Li-Zelevinsky, we can prove the bijectivity of the map $\Phi$.

\begin{proof}[Proof of \autoref{thm: Phi bijectivity}]
By \autoref{thm: bij forward} and \autoref{lem: bij inverse}, we can see that $\Phi$ is a bijection from $\calF'(\calD_n)$  onto the set of compatible pairs in $\calC_n$.  Using the work of Lin \cite[Proposition 4.7.3]{Lin}, we additionally see that $\Phi$ is weight-preserving.  
\end{proof}

\section{Quantization of Colored Dyck Subpaths}\label{sec: quantum}
In Lee and Schiffler's work on expanding rank-$2$ cluster variables, they showed that the coefficients of the Laurent expansion could be obtained by taking sums over certain collections of colored Dyck subpaths.  Each such collection was taken to have weight $1$.  In order to quantize this construction, we instead weight each collection by a power of a formal variable $q$.  We then show that an analogous formula holds for rank-$2$ quantum cluster variables with an appropriate choice of $q$-weights, where setting $q=1$ recovers Lee and Schiffler's formula.

As discussed by Lee-Li-Rupel-Zelevinky \cite[Section 3]{LLRZ}, the combinatorial formula for greedy basis elements of a rank-$2$ cluster algebra cannot be extended to the quantum setting by merely weighting compatible pairs by a power of $q$, since positivity of these elements can fail in general. However, Dylan Rupel \cite[Corollary 5.4]{Rup2} established that the classical rank-$2$ formula for the quantum cluster variables, which are a proper subset of the greedy basis, given by Lee-Li-Zelevinsky \cite{LLZ} could be extended in this way.  We proceed by applying the bijection established in the previous section to Rupel's expansion formula over weighted compatible pairs associated to quantum cluster variables.  

An advantage of \autoref{thm: quantum expansion} is that it requires fewer computations compared to Rupel's formula in \cite[Corollary 5.4]{Rup2}.  Our formula requires $O(|\beta|^2)$ computations, where $|\beta|$ is the number of subpaths in $\beta \in \calF'(\calD_n)$.  Rupel's formula requires $\binom{c_{n-1} + c_{n-2}}{2}$ computations, which is generally much larger.  Moreover, without knowledge of the bijection between collections of colored subpaths and compatible pairs, it is unclear how to generate all compatible pairs.  Na{\"i}vely, one must consider all collections of edges of $\calC_n$ and check that the compatibility condition holds for each pair of edges.  It thus seems more efficient to generate all collections in $\calF'(\calD_n)$ and compute their quantum weights using \autoref{def: quantum weight} than to generate all compatible pairs on $\calC_n$ and compute quantum weights using \cite[Corollary 5.4]{Rup2}.

\begin{figure}
\centering
\includegraphics[width = 3in]{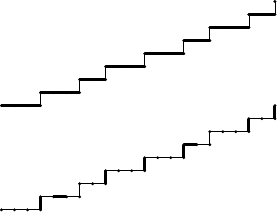}
\caption{The top path depicts the compatible pair on $\calC_6$ obtained by applying $\Phi$ to $\beta_\emptyset$ on $\calD_6$, which has $S_1 = \{\eta_1,\eta_2,\dots,\eta_{21}\}$ and $S_2 = \emptyset$.  The bottom path depicts the compatible pair obtained by applying $\Phi$ to the collection of colored Dyck subpaths shown in \autoref{fig: Phi map}.}
\label{fig: quantum weight}
\end{figure}

In our proof of \autoref{thm: quantum expansion}, we translate each compatible pair into a finite word so that we can refer to the language of combinatorics on words.  We will work over the alphabet $A = \{h,v,H,V\}$, with $A^*$ denoting the set of finite words on $A$.  For the purposes of this section, we represent a compatible pair by a word in $A^*$.  The letters $h$ and $H$ (resp., $v$ and $V$) represent horizontal (resp., vertical) edges, with the capital letter denoting those edges in $S_1$ (resp., $S_2$).  

\begin{exmp}\label{exmp: compatible word}
The word in $A^*$ corresponding to the top compatible pair in \autoref{fig: quantum weight} is 
$$H^3vH^3vH^2vH^3vH^3vH^2vH^3vH^2v\,.$$

The word in $A^*$ corresponding to the bottom compatible pair in \autoref{fig: quantum weight} is  
$$h^3VhHhvh^2Vh^3Vh^3VHhvh^3Vh^2V\,.$$
\end{exmp}

We define a morphism $f: \Z A^* \to \Z$, where $\Z A^*$ is the group of formal $\Z$-sums of words in $A^*$.  The function $w_q$ is defined on words of length $2$ in $A^*$ as follows:
$$w_q(hv) = w_q(Hv) = w_q(hV) = 1\,, \;\;\; w_q(Hh) = w_q(vV) = r\,, \;\;\; w_q(VH) = r^2 - 1\,,$$
and for $x,y \in A$, $w_q(xy) = - w_q(yx)$.  Note that in particular, this implies that\\ ${w_q(hh) = w_q(HH) = w_q(vv) = w_q(VV) = 0}$.
The function $w_q$ naturally extends to formal $\Z$-sums of words of length $2$ on $A$.  It is then extended to words of larger length by taking the formal sum over all length $2$ (not necessarily contiguous) subwords with multiplicity, and again extended naturally to formal $\Z$-sums of any words on $A$.  We sometimes apply $w_q$ to a compatible pair; in this case, we interpret $w_q$ as being applied to the corresponding word on $A$.  We refer to $w_q$ as the \emph{quantum weight} of a word or compatible pair.  Computing the value of $w_q$ on the word associated to a compatible pair in this way is essentially calculating Rupel's weighting on compatible pairs associated to quantum cluster variables \cite{Rup2}.

\begin{exmp}\label{exmp: quantum pair calculation}
Let $t_1$ (resp., $t_2$) denote the word in $A^*$ corresponding to the top (resp., bottom) compatible pair in \autoref{fig: quantum weight}, computed in \autoref{exmp: compatible word}.  Note that here we have $r = 3$.  Looking at all the length-$2$ subwords of $t_1$, we can see that $t_1$ has $98$ instances of the subword $Hv$ and $70$ instances of the subword $vH$.  The only other length-$2$ subwords of $t_1$ are $HH$ and $vv$, and we have $w_q(HH) = w_q(vv) = 0$.  Thus we can compute
$$w_q(t_1) = 98 w_q(Hv) + 70 w_q(vH) = 98\cdot 1 + 70 \cdot (-1) = 28\,.$$
Via a similar computation, we have
\begin{align*}
    w_q(t_2) &= 69 w_q(hV) + 45 w_q(Vh)+ 21w_q(Hh) + 17w_q(hH)  + 7w_q(vV) + 5w_q(Vv) \\&\;\;\; + 5w_q(VH) + 7w_q(HV)  + 19w_q(hv) + 19(vh) + 3w_q(Hv) + w_q(vH)\\
    &= 69 - 45 + 21\cdot 3 - 17 \cdot 3 + 7 \cdot 3 - 5 \cdot 3  + 5\cdot 8 - 7 \cdot 8 + 19 - 19 + 3 - 1\\
    &= 28
\end{align*}
As we will show, there are more efficient methods for computing these weights than looking at all length-$2$ subwords.
\end{exmp}

For a word $u \in A^*$ and a letter $x \in A$, let $(u)_x$ denote the number of instances of $x$ in $u$.  In order to find a compact formula for the weights corresponding to collections of subpaths

\begin{lem}\label{lem: quantum empty path}
Let $\beta_\emptyset \in \calF'(\calD_n)$ be the empty collection of colored Dyck subpaths on $\calD_n$.  For all $n \geq 3$, we have $$w_q\left(\Phi_n(\beta_\emptyset)\right) = c_{n-1}+c_{n-2}-1\,.$$
\end{lem}
\begin{proof}
We proceed by induction on $n$.  For the base case $n = 3$, we have 
$$w_q\left(\Phi_3(\beta_\emptyset)\right) = w_q(H^rv) = rw_q(Hv) + \binom{r}{2}w_q(HH) = r = c_3 + c_2 - 1\,.$$

We now proceed to the inductive step.  Let $\psi$ be the morphism $\{H \mapsto Hv\}$.  Observe that for a word $u \in \{H,v\}^*$ that starts with $H$, ends with $v$, and has no consecutive instances of $v$, we have
$$w_q(\psi(u)) = w_q(u) + (u)_H\,.$$ 
Let $u_j$ be the word associated to $\Phi_j(\beta_\emptyset)$. Then by \autoref{rem: C to D} and \autoref{lem: corner map}, $u_{n+1}$ can be obtained by applying the morphism $\chi = \{H \mapsto H^rv, v \mapsto H^{r-1}v\}$ to $\psi^{-1}(u_n)$.  We can readily calculate
\begin{align*}
    &w_q(\chi(HH)) = w_q(H^rvH^rv) = 2r = w_q(HH)+ 2w_q(\chi(H))\,,\\
    &w_q(\chi(vv)) = w_q(H^{r-1}vH^rv) = 2r-2 = w_q(vv) + 2w_q(\chi(v))\,,\\
    &w_q(\chi(Hv)) = w_q(H^rvH^{r-1}v) = 2r = w_q(Hv) + w_q(H^rv) + w_q(H^{r-1}v)\,, \\
    &w_q(\chi(vH)) = w_q(H^{r-1}vH^{r}v) = 2r-2 = w_q(vH) + w_q(H^{r-1}v) + w_q(H^rv)\,. 
\end{align*}
Moreover, $\psi^{-1}(u_n)$ has $c_{n-1}-c_{n-2}$ instances of $H$ and $c_{n-2}$ instances of $v$.
Therefore, we have
\begin{align*}
    w_q(\Phi_{n+1}(\beta_\emptyset) &= w_q(\chi(\psi^{-1}(u_n)\\
    &= w_q(\psi^{-1}(u_n)) + (c_{n-1}-c_{n-2})w_q(\chi(H)) + c_{n-2}w_q(\chi(v))\\
    &= \left(\psi(u_n) - c_{n-2}\right) + r(c_{n-1}-c_{n-2}) + (r-1)c_{n-2}\\
    &= \left(c_{n-1}+c_{n-2} - 1 - c_{n-2}\right) + c_n\\
    &= c_n + c_{n-1} + 1\,.\qedhere
\end{align*}
\end{proof}

Having calculated the weight of the empty collection of paths in $\calF'(\calD_n)$, we wish to calculate the quantum weight when we add in colored subpaths.  Note that for the word corresponding to the compatible pair, this involves swapping out certain instances of $H$ for $h$ and $v$ for $V$.  We now calculate how such a substitution affects the quantum weight of the word.  

\begin{lem}\label{lem: quantum coord formula}
Let $t_1,u_1,t_2,u_2,\dots,t_s,u_s$ be words corresponding to compatible pairs.  Furthermore, suppose that $t_i,u_i \in \{H,v\}$.  Let $\sigma$ be the morphism $\{H \mapsto h, v \mapsto V\}$.  Then we have
$$w_q\left(\prod_{i=1}^s t_i \sigma(u_i)\right)  = w_q\left(\prod_{i=1}^s t_i u_i\right) + \sum_{i=1}^s\sum_{j=1}^s (-1)^{\mathbbm{1}_{i < j}} \bigg(r(t_j)_H\cdot(u_i)_h + \big(r(t_j)_v - r^2(t_j)_H\big)\cdot(u_i)_V\bigg)$$
\end{lem}
\begin{proof}
Since $w_q(hV) = w_q(Hv)$, we have $w_q(u_i) = w_q\left(\sigma(u_i)\right)$ for all $i$.  Hence we need only to calculate the change its value under $w_q$ after applying $\sigma$ to the length-$2$ subwords with one letter from a $u_i$ and the other from a $t_j$.

Let $U_i$ denote the value under $w_q$ of the sum over all length-$2$ subwords of $\prod_{i=1}^s t_i \sigma(u_i)$ with one letter in $\sigma(u_i)$ and the other from some $t_j$.   We then have
\begin{align*}
    U_i &= \sum_{j=1}^s (-1)^{\mathbbm{1}_{i < j}} \bigg((t_j)_H\cdot(u_i)_h\cdot\left(w_q(Hh) - w_q(HH)\right) + (t_j)_H\cdot(u_i)_V\cdot\left(w_q(HV) - w_q(Hv)\right)\\
    &\quad\quad\quad\quad + (t_j)_v\cdot(u_i)_h\cdot\left(w_q(vh) - w_q(vH)\right) +   (t_j)_v\cdot(u_i)_V\cdot\left(w_q(vV) - w_q(vv)\right)\bigg)\\
    &= \sum_{j=1}^s (-1)^{\mathbbm{1}_{i < j}} \bigg(r\big((t_j)_H\cdot(u_i)_h + (t_j)_v\cdot(u_i)_V\big) - r^2(t_j)_H\cdot(u_i)_V\bigg)\,.\qedhere
\end{align*}
\end{proof}

Applying the previous general result about compatible pairs to the specific case of $\calC_n$ and using the connection between $\calF'(\calD_n)$ and compatible pairs on $\calC_n$, we can now prove the quantum cluster variable expansion formula.

\begin{proof}[Proof of \autoref{thm: quantum expansion}]
Adding a path to $\beta \in \calF'(\calD_n)$ corresponds to applying the morphism $\sigma$ from \autoref{lem: quantum coord formula} to the appropriate portion of associated compatible word.
Note that $|\beta_i|_2 = (\beta_i)_h$ and $|\beta_i|_1 = (\beta_i)_V$.  Similarly, $|\overline{\beta_j}|_2 = (\overline{\beta_j})_H$ and $|\overline{\beta_j}|_1 = (\overline{\beta_j})_v$.  It follows from \autoref{lem: quantum coord formula}, \autoref{lem: quantum empty path}, and the definition of $w_q$ for words in $A^*$ that $w_q(\beta) = \gamma_\omega + \beta_\omega$, where the terms on the right hand side are those in \cite[Corollary 5.4]{Rup2}.  Thus expansion formula can be deduced directly from \cite[Corollary 5.4]{Rup2}.  
\end{proof}

\begin{exmp}
Let $\beta = \{\beta_1,\dots,\beta_6\} \in \calF'(\calD_6)$ be the collection of colored Dyck subpaths shown in \autoref{fig: Phi map}.  Then we have 
$\overline{\beta_0} = \overline{\beta_1} = \overline{\beta_6} = \emptyset$, $\overline{\beta_2} = \{\alpha_4\}$, $\overline{\beta_3} = \{v_2\}$, $\overline{\beta_4} = \{\alpha_{14}\}$, and $\overline{\beta_5} = \{v_6\}$. Applying \autoref{thm: quantum expansion}, we have
\begin{align*}
    w_q(\beta) = (c_5+c_4 - 1) + 3(15-4)-9(5-1)+3(5-1)+3(6-13)-9(2-4) + 3(2-4) = 28\,.
\end{align*}
which confirms the second calculation in \autoref{exmp: quantum pair calculation}.
\end{exmp}

\section{Further Directions}\label{sec: further directions}
Many of our methods rely on the highly structured nature of the maximal Dyck paths $\calC_n$ and $\calD_n$.  We use this to better understand the conditions for a set of edges to form a compatible pair on $\calC_n$, in particular deriving a criterion for compatibility in terms of the sequences of consecutive vertical edges.  While $\calC_n$ is the relevant choice of maximal Dyck path for the cluster variables, compatible pairs over arbitrary maximal Dyck paths were studied by Lee, Li, and Zelevinsky \cite{LLZ} in their work on the greedy basis. One way to study compatibility is in terms of forbidden edge sets, i.e., the minimal subsets of edges which violate the compatibility condition for compatible pairs.  It is easy to verify that for the staircase Dyck path $\calP(a,a)$, a set of edges is compatible if and only if it does not contain a horizontal edge and the vertical edge immediately following it.  From the proof of the bijectivity of $\Phi$, it follows that on $\calC_n$ the forbidden edge sets are the images under $\Phi$ of an $(m,w)$-brown path and the $c_{m-1}-wc_{m-2}$ edges preceding it. It could be interesting to study whether the criterion for compatibility can also be reduced for other families of maximal Dyck paths.  

While positivity fails in general for the quantum greedy basis \cite[Section 3]{LLRZ}, it would be interesting to know under what conditions positivity holds.  Rupel's formula \cite[Corollary 5.4]{Rup2} establishes the positivity property for the quantum cluster variables, but perhaps this is a special case of a more general phenomenon.  If so, these elements may also admit a quantum weighting of the associated compatible pairs, similar to that given by Rupel.

\section{Acknowledgements}
The author thanks her advisor, Lauren Williams, for introducing her to the various interesting expansion formulas for low-rank cluster algebras and for discussions at many stages of this research.  The author is also grateful to Kyungyong Lee, Feiyang Lin, Gregg Musiker, and Ralf Schiffler for their helpful correspondence and comments about this work.  She additionally thanks Kyungyong Lee for suggesting the generalization of \autoref{cor: LS modified expansion} to the case of cluster algebras with coefficients.  A portion of this work was completed while the author was visiting l'Institut de Recherche en Informatique Fondamentale, and the author extends her gratitude to Sylvie Corteel for hosting her.  The author was supported by NSF Graduate Research Fellowship and a Jack Kent Cooke Foundation Graduate Fellowship.

\end{document}